\documentclass{amsart}

\copyrightinfo{2008}
  {American Mathematical Society}

\usepackage{amssymb}

\newcommand{\C}			{\mathbb C}
\newcommand{\R}			{\mathbb R}
\newcommand{\Q}			{\mathbb Q}
\newcommand{\U}			{\mathbf U}
\newcommand{\f}			{\mathbf f}
\newcommand{\g}			{\mathbf g}
\newcommand{\h}			{\mathbf h}
\newcommand{\p}			{\mathbf p}
\newcommand{\q}			{\mathbf q}
\newcommand{\e}			{\mathbf e}
\newcommand{\dd}		{\mathbf d}
\newcommand{\rr}		{\mathbf r}
\newcommand{\vv}		{\mathbf v}
\newcommand{\map}		{\longrightarrow}
\newcommand{\Lin}		{\mathcal L}
\newcommand{\Linfty}	{L^\infty}
\newcommand{\Ltwo}		{L^2}
\newcommand{\Lone}		{L^1}
\newcommand{\Lp}		{L^p}
\newcommand{\grad}		{\nabla}
\newcommand{\cross}		{\times}
\newcommand{\curl}		{\operatorname{curl}}
\newcommand{\supp}		{\operatorname{supp}}

\newtheorem{theorem}{Theorem}[section]
\newtheorem{lemma}[theorem]{Lemma}
\newtheorem{proposition}[theorem]{Proposition}
\newtheorem{corollary}[theorem]{Corollary}
\newtheorem{conjecture}[theorem]{Conjecture}
\newtheorem{mconjecture}[theorem]{Meta-Conjecture}

\numberwithin{equation}{section}

\begin{document}

\title{On the unfolding of simple closed curves}

\author{John Pardon}
\address{Durham Academy Upper School\\3601 Ridge Road\\Durham, North Carolina 27705}
\curraddr{Princeton University\\Princeton, New Jersey 08544}
\email{jpardon@princeton.edu}

\subjclass[2000]{Primary 53C24; Secondary 53A04}

\date{January 2, 2007}

\begin{abstract}
I show that every rectifiable simple closed curve in the 
plane can be continuously deformed into a convex 
curve in a motion which preserves arc length and does 
not decrease the Euclidean distance between any pair of 
points on the curve.  This result is obtained by approximating 
the curve with polygons and invoking the result of 
Connelly, Demaine, and Rote that such a motion exists 
for polygons.  I also formulate a generalization of their 
program, thereby making steps toward a fully continuous 
proof of the result.  To facilitate this, I generalize 
two of the primary tools used in their program: 
the Farkas Lemma of linear programming to Banach 
spaces and the Maxwell-Cremona Theorem of rigidity theory 
to apply to stresses represented by measures on the plane.
\end{abstract}

\maketitle

\section{Introduction}\label{intro}

Imagine a loop of string lying flat on a table without 
crossing itself.  Now suppose the loop is slowly 
deformed until it becomes convex, without stretching 
or breaking it, in an {\it expansive} motion.  
By expansive, I mean that 
if you pick any pair of points on the string, then during 
the deformation, the distance between them will be 
nondecreasing.  Then we can ask whether, given 
an initial loop, there always exists an expansive motion 
which deforms that loop until it becomes convex.  If the 
loop is a polygon, then the answer is yes, as proved by 
Connelly, Demaine, and Rote \cite{cdr}.  The first theorem 
of this paper (Theorem \ref{main}) is that the answer is yes for any rectifiable 
curve, no matter how complicated (in section \ref{path}, 
we give some examples of pathological curves to which the 
theorem applies).  This solves Problem 4 listed by Ghomi 
\cite[p. 1]{open}.

My proof of the main theorem uses a limiting process, 
relying on the result of \cite{cdr}.  I next generalize 
the program used in \cite{cdr}, which 
relies on techniques of linear programming, specifically the Farkas 
Lemma.  This approach naturally lends itself to computation; 
an example of research on the computation of nonexpansive unfoldings 
of polygons is given by \cite{energy}.  In my continuous 
analogue of the program, I develop a version of the Farkas 
Lemma for Banach spaces (Theorem \ref{farkas}) as well as a 
continuous version of the Maxwell-Cremona Theorem (Theorem 
\ref{maxwell}), a combinatorial version of which was used in 
the program in \cite{cdr}.  A different version of the Farkas Lemma 
in Banach spaces and specifically in $\Lp$ spaces has been studied 
in \cite{farkas}.  I am not aware of any previous 
generalization of the Maxwell-Cremona Theorem to the 
case I consider here.  Finally, I use the continuous 
version of the program to give a different proof of the existence 
of infinitesimal expansions for polygons.  The hope is that 
a continuous analogue of the discrete program could yield 
a direct proof (one which does not rely on approximation by 
polygons) of the main theorem for some class of curves more 
general than polygons.

I would like to thank Robert Bryant for many useful 
conversations about the work in this paper, regarding both 
its content and presentation, and Robert Connelly for 
suggesting some reogranization to clarify the results.  
I also thank Andrew Ferrari for introducing me to many of 
the techniques used here.

\subsection{Notation}

We will use the following function spaces:
\begin{description}
\item[$C(X,Y)$] the Banach space of continuous functions from 
$X$ to $Y$ given the supremum norm.
\item[$C_c(X,Y)$] the subspace of $C(X,Y)$ consisting of 
functions of compact support.
\item[$C_0(X,Y)$] the Banach space completion of $C_c(X,Y)$ 
with respect to the supremum norm.  These are the functions 
that ``vanish at infinity''.
\item[$C_0^\infty(X,Y)$] the subspace of $C_c(X,Y)$ consisting 
of infinitely differentiable functions.
\item[$\Lp(X,Y)$] the Banach space of $\Lp$ functions from 
$X$ to $Y$.
\end{description}
If $Y$ is left out, it is assumed to be $\R$, except in 
section \ref{cpx}, where it is assumed to be $\C$.  
All Hilbert and Banach spaces are implicitly assumed to 
be over $\R$, except in section \ref{cpx}, where they will 
be over $\C$.  If $E$ is a Banach space, $E^*$ is its dual.  
The duality bracket $\langle x,y\rangle$ will be used both 
in the case that $x\in E^*$ and $y\in E$, and in the case 
that $x,y\in H$, a Hilbert space.  We will write $\Lin(X,Y)$ 
for the Banach space of bounded linear transformations from 
$X$ to $Y$ given the operator norm.

\section{Proof for General Curves using \cite{cdr}}\label{cpx}

\subsection{Preliminaries}

Consider a simple closed curve in the plane.  I wish to prove 
the existence of a continuous deformation of the curve into 
a convex curve, so that the intrinsic distance between every 
pair of points on the curve stays constant, and the extrinsic 
distance between every pair of points on the curve is 
nondecreasing.  Here, by intrinsic distance I mean the distance 
along the curve, and by extrinsic distance I mean the Euclidean 
distance in $\R^2$.

A curve is called {\it rectifiable} if a finite intrinsic 
distance can be defined between every pair of points, that is, 
the supremum of the lengths of all inscribed 
polygons is finite:
\begin{equation}
L_x^y(\f):=\sup_{x=a_0<a_1<\cdots<a_k=y}\sum_{j=1}^k|\f(a_j)-\f(a_{j-1})|<\infty
\end{equation}
We will only consider rectifiable curves in 
this paper.  If a curve is rectifiable, then it has a {\it unit 
speed parameterization}, that is $\f(s)=\int_0^s\f'(s')\,ds'$ and 
$|\f'(s)|=1$ almost everywhere.  Since a homothety will scale the 
arc length of a curve, it suffices to consider simple closed 
curves of length $2\pi$.  Thus, given $\f_0$, we seek a 
continuous family of simple closed curves $\f_t:\R/2\pi\map\R^2$ 
parameterized by $t\in[0,1]$ such that each curve is of unit 
speed, $|\f_{t_1}(x)-\f_{t_1}(y)|\leq|\f_{t_2}(x)-\f_{t_2}(y)|$ 
whenever $t_1\leq t_2$, and $\f_1$ is convex.

\subsection{Main Result}

For this section, it will be natural to consider curves 
in $\C$ (rather than $\R^2$).  Thus Banach spaces will be 
over $\C$.  It will be convenient to have our curves 
reside in the following space:
\begin{equation}
\mathcal D:=\left\{\f:\R/2\pi\map\C\Bigm|
\text{$\f(0)=0$, $\f$ absolutely continuous, $\f'\in\Linfty(\R/2\pi)$}\right\}
\end{equation}
There is, of course, the natural correspondence between 
$\f\in\mathcal D$ and $\f'\in\{u\in\Linfty(\R/2\pi):\int u=0\}$.  
Thus $\mathcal D$ is a Banach space with norm $\|\f'\|_\infty$.  Now 
topologize $\mathcal D$ using the weak-$*$ topology on 
$\Linfty(\R/2\pi)$.  Since $\Lone(\R/2\pi)$ is separable, 
the Banach-Alaoglu Theorem implies that any norm bounded 
sequence in $\mathcal D$ has a convergent subsequence.  
The choice of topology on $\mathcal D$ is justified by 
the following lemma.

\begin{lemma}\label{unif}
Suppose $\f_n\to\f$ in $\mathcal D$, then $\f_n\to\f$ uniformly.
\end{lemma}

\begin{proof}
By the Uniform Boundedness Principle, we know that 
$\|\f_n\|$ is bounded.  Thus there exists $M$ 
with $|\f_n'|\leq M$, hence $\{\f_n\}$ is an equicontinuous family.  
It is clear that $\f_n\to\f$ pointwise since we have 
$\int\chi_{[0,x]}\f_n'\to\int\chi_{[0,x]}\f'$.  And 
an equicontinuous sequence of functions converges 
pointwise if and only if it converges uniformly.
\end{proof}

Define the continuous function $\mathcal E:\mathcal D\map\R$ by 
$\mathcal E(\f)=\iint_{(\R/2\pi)^2}|\f(x)-\f(y)|$.  Also define the following 
order relation on $\mathcal D$: we say that $\f\trianglelefteq\g$ 
if and only if $|\f(x)-\f(y)|\leq|\g(x)-\g(y)|$ 
for all $x$ and $y$.

\begin{theorem}\label{main}
Given a unit speed simple closed curve $\f:\R/2\pi\map\C$, 
there exists a continuous function $\h:[0,1]\map\mathcal D$ such 
that:
\begin{itemize}
\item[(1)] $\h(0)=\f$.
\item[(2)] $\h(1)$ is convex.
\item[(3)] $\h(t)$ has unit speed for all $t$.
\item[(4)] If $t_1\leq t_2$, then $\h(t_1)\trianglelefteq\h(t_2)$.
\end{itemize}
\end{theorem}

\begin{proof}
For $n\geq 3$, consider the polygon $\mathcal P_n$ 
inscribed in $\f$ which has $n$ vertices spaced 
out at multiples of $2\pi/n$ starting at zero.  
Explicitly:
\begin{equation}
\mathcal P_n(x):=\left(1-\left\{\frac{nx}{2\pi}\right\}\right)
\f\left(\frac{2\pi}n\left\lfloor\frac{nx}{2\pi}\right\rfloor\right)
+\left\{\frac{nx}{2\pi}\right\}
\f\left(\frac{2\pi}n\left(\left\lfloor\frac{nx}{2\pi}\right\rfloor+1\right)\right)
\end{equation}
This polygon may or may not be simple.  It will, 
however, divide the plane into a finite number of 
simply connected regions.  Let $\mathcal P_n'$ be a constant 
speed $s_n\leq 1$ parameterization of the boundary 
of that region which has greatest area.  Then let 
$\h_n:[0,1]\map\mathcal D$ be continuous and satisfy:
\begin{itemize}
\item[(1$'$)] $\h_n(0)=\mathcal P_n'$.
\item[(2$'$)] $\h_n(1)$ is convex.
\item[(3$'$)] $\h_n(t)$ has speed $s_n\leq 1$ for all $t$.
\item[(4$'$)] If $t_1\leq t_2$, then $\h_n(t_1)\trianglelefteq\h_n(t_2)$.
\item[(5$'$)] $\h_n(t)(\pi)\in\R_{>0}$ for all $t$.
\item[(6$'$)] $\mathcal E(\h_n(t))$ is a linear function of $t$.
\end{itemize}
The existence of an $\h_n$ satisfying (1$'$)--(4$'$) is 
implied by Theorem 1 of \cite[p. 207]{cdr}.  
Condition (5$'$) can be achieved by properly rotating 
each curve.  The motion of \cite{cdr} is {\it strictly} 
expansive, so $\mathcal E(\h_n(t))$ will be strictly 
increasing, so a simple reparameterization in $t$ 
suffices to make it linear and satisfy (6$'$).

Let $\Q_{[0,1]}=\Q\cap[0,1]$.  
This set is countable; suppose $\{r_i\}_{i=1}^\infty$ 
is a counting of it.  Let $\h_n^{(0)}=\h_n$.  
Inductively, let $\h_n^{(i)}$ be a subsequence of 
$\h_n^{(i-1)}$ such that $\h_n^{(i)}(r_i)$ converges.  
(Such a subsequence is guaranteed to exist since 
$\|\h_n^{(i)}(r_i)\|=s_n$ is bounded).  Now $\h_j^{(j)}$ 
converges pointwise to a function $\tilde{\h}:\Q_{[0,1]}\map\mathcal D$ 
which satisfies:
\begin{itemize}
\item[(1$''$)] $\tilde{\h}(0)(\R/2\pi)=\f(\R/2\pi)$.
\item[(2$''$)] $\tilde{\h}(1)$ is convex.
\item[(3$''$)] $\tilde{\h}(t)$ has speed $\leq 1$.
\item[(4$''$)] If $t_1\leq t_2$, then $\tilde{\h}(t_1)\trianglelefteq\tilde{\h}(t_2)$.
\item[(5$''$)] $\tilde{\h}(t)(\pi)\in\R_{>0}$ for all $t$.
\item[(6$''$)] $\mathcal E(\tilde{\h}(t))$ is a linear function of $t$.
\end{itemize}
We will now construct $\h:[0,1]\map\mathcal D$.  For every 
$t\in[0,1]$, we set $\h(t)$ to be some arbitrary subsequential 
limit of $\tilde{\h}(q_j)$ where $q_j$ is some sequence of 
rationals converging to $t$.  Clearly $\h$ satisfies 
(1$''$)--(6$''$) as well.  Now (1$''$) and (3$''$) together 
mean that $\h(0)(s)=\f(s+\Delta)$ for some 
$\Delta$.  We can take $\Delta=0$.  Hence we have 
(1), (2), and (4).  To prove (3), note that:
\begin{equation}
\left|\frac{\f(x+h)-\f(x)}h\right|\leq\left|\frac{\h(t)(x+h)-\h(t)(x)}h\right|\leq 1
\end{equation}
As $h\to 0$, the left hand side approaches $1$ for 
almost all $x$, hence $|\h(t)'(x)|=1$ almost 
everywhere as desired.

Finally, we must show that $\h$ is in fact continuous.  
This follows from (5$''$) and (6$''$) in the following way.  
Suppose the contrary, that there is some $t$ where 
$\h$ is not continuous.  Then there exists a sequence 
$q_j\to t$ with either $q_j<t$ for all $j$ or $q_j>t$ 
for all $j$, and a neighborhood $N$ of $\h(t)$ such that
$\h(q_j)\notin N$ for all $j$.  Now a subsequence of 
$\h(q_j)$ will converge in $\mathcal D$ to a limit $\g$.  
Now we have:
\begin{itemize}
\item[(i)] $\mathcal E(\g)=\mathcal E(\h(t))$
\item[(ii)] $\g\trianglelefteq\h(t)$ or $\g\trianglerighteq\h(t)$ 
depending on whether $q_j<t$ or $q_j>t$
\item[(iii)] $\g(0)=0=\h(t)(0)$
\item[(iv)] $\g(\pi),\h(t)(\pi)\in\R_{>0}$
\end{itemize}
The conditions (i) and (ii) imply 
that $|\g(x)-\g(y)|=|\h(t)(x)-\h(t)(y)|$ for 
all $x$ and $y$.  This means that the curves are 
rigid motions of each other.  Then (iii) and (iv) 
imply that they are actually the same 
curve since they have the same orientation.  Thus a 
subsequence of $\h(q_j)$ converges to $\h(t)$.  This 
is of course a contradiction since each $\h(q_j)$ 
is outside the neighborhood $N$ of $\h(t)$.  
This contradiction proves that $\h$ is continuous.
\end{proof}

\subsection{Pathological Rectifiable Curves}\label{path}

Define $f_-$ and $f_+$:
\begin{equation}
f_\pm(x)=\begin{cases}x^2\sin x^{-1}\pm e^{-1/x}&x>0\cr 0&x=0\end{cases}
\end{equation}
If we plot $f_-$ and $f_+$ on $[0,\pi^{-1}]$ and add 
line segments around the left side of the curve to close 
it, we get an infinite number of interlocking ``teeth''.  
This example is based on a polygon with a finite number 
of such teeth unfolded by Erik Demaine.  We also have:
\begin{equation}
g(t)=\begin{cases}t^2e^{i/t}&t>0\cr 0&t=0\cr-t^2e^{-i/t}&t<0\end{cases}
\end{equation}
Plotting $g$ on $[-\pi^{-1},\pi^{-1}]$ and adding line 
segments to close the curve gives a simple closed curve 
with an infinite spiral.  By Theorem \ref{main}, 
both of these curves can be unfolded in an expansive 
motion, something which is not at all intuitive considering 
their geometry.

\section{A Generalization of the CDR Program}

The program in \cite{cdr} proves the existence of an infinitesimal 
expansion for any polygon.  That is, if a nonconvex polygon has 
verticies $\p_i$, it shows the existence of velocities $\vv_i$ satisfying:
\begin{align}
(\p_i-\p_{i+1})\cdot(\vv_i-\vv_{i+1})&=0\\
(\p_i-\p_j)\cdot(\vv_i-\vv_j)&>0\text{ for $i$ and $j$ not adjacent}
\end{align}
From this, it is relatively straightforward to solve a 
differential equation of the form $\frac d{dt}\{\p_i\}=\{\tilde{\vv}_i\}$ 
(where the $\{\tilde{\vv}_i\}$ depend continuously on the $\{\p_i\}$), 
thus constructing an expansive motion of the polygon.  Clearly, 
if we have a curve $\f$, then the analogue is to find a variation 
$\varphi$ satisfying:
\begin{align}
\f'(x)\cdot\varphi'(x)&=0\text{ for all $x$}\\
(\f(x)-\f(y))\cdot(\varphi(x)-\varphi(y))&\geq 0\text{ for all $x$ and $y$}
\end{align}

The generalized program developed here will be able to 
prove the existence of infinitesimal expansions for 
polygons, a hard theoretical result of \cite{cdr}.  
It also proves the existence of ``almost'' expansive variations 
for all rectifiable curves which in a neighborhood of any point 
look like the rotated graph of a function from $\R$ to $\R$.  
By this I mean that for every $x\in\R/2\pi$, there exists 
$\vv\in\R^2$ such that $\f(y)\cdot\vv$ is one to one in a 
neighborhood of $x$.  The final result of this generalized 
program is Theorem \ref{gen}.

The generalizations of the Farkas Lemma and the Maxwell-Cremona 
Theorem, the tools used in the program, are stated and proved 
in sections \ref{farkass} and \ref{maxwells} respectively.

\subsection{Notation}

Let $H:=\{u\in C(\R/2\pi,\R^2):u(0)=0$, $u$ absolutely 
continuous, and $\int|u'|^2<\infty\}$.  So that $H$ is a Hilbert 
space, equip it with the norm $\sqrt{\int|u'|^2}$ 
and inner product $\int u'\cdot v'$.  Topologize $H$ with 
the weak topology.  We will need the sets:
\begin{align}
Q_\f&:=\{u\in H:u'\cdot\f'\equiv 0\}\text{ (a closed subspace)}\\
T&:=\{t\in C((\R/2\pi)^2)^*:t\geq 0\}
\end{align}
Note that we will be looking for $\varphi\in Q_\f$, since it is 
these variations which preserve arc length.  Also note that in 
this section, we do not assume that $\f$ is parameterized by 
arc length.

\begin{lemma}\label{unif2}
If $\g_n\to\g$ in the weak topology on $H$, then 
$\g_n\to\g$ uniformly.
\end{lemma}

\begin{proof}
This is completely analogous to Lemma \ref{unif}.  We know that 
$\g_n\to\g$ pointwise.  Observing that $\|\g_n\|$ is bounded, we 
have the inequality:
\begin{equation}
\int_a^b|\g_n'|=\int_{\R/2\pi}\g_n'\frac{\g_n'}{|\g_n'|}\chi_{[a,b]}\leq
\sqrt{\int_{\R/2\pi}|\g_n'|^2}\sqrt{\int_a^b\left|\frac{\g_n'}{|\g_n'|}\right|^2}
\leq M\sqrt{b-a}
\end{equation}
This shows that $\g_n$ are uniformly continuous, and hence 
converge uniformly.
\end{proof}

Define $D\subset H$, the set of curves we will consider, 
to be the set of $\f\in H$ satisfying:
\begin{itemize}
\item[(1)] $\f$ is a simple closed curve, that is $\f$ is injective.
\item[(2)] $\f'\ne 0$ almost everywhere (this in fact is not implied by (1)).
\item[(3)] For every $x$, there exists $\delta>0$ and $\vv$ 
such that $\f(y)\cdot\vv$ is one to one for $|y-x|<\delta$. (locally graph-like)
\end{itemize}
The symbol $\f$ will always denote a member of $D$.

The following bounded operator will be essential 
to the program; it is called the {\it Rigidity 
Operator}:
\begin{gather}
R_\f:H\map C((\R/2\pi)^2)\cr
(R_\f\varphi)(x,y)=(\f(x)-\f(y))\cdot(\varphi(x)-\varphi(y))
\end{gather}

\subsection{Outline of the Program}

Before we state the final result of the program in its full 
generality (Theorem \ref{gen}), it is useful to state the 
following corollary which gives the general idea of the result.

\begin{corollary}\label{cgen}
Let $\f\in D$ not be convex.  Let $V\subset(\R/2\pi)^2$ be 
closed and have the property that for all $(x,y)\in V$, 
the line segment between $\f(x)$ and $\f(y)$ is not competely 
contained in $\f(\R/2\pi)$ (for example, if $\f$ has no 
straight sections, we can take $V=\{(x,y)\in(\R/2\pi)^2:|x-y|>\epsilon\}$).  
Then there exists a $\varphi\in Q_\f$ such that:
\begin{equation}\label{expan}
(\f(x)-\f(y))\cdot(\varphi(x)-\varphi(y))>0\text{ for all $(x,y)\in V$}
\end{equation}
\end{corollary}

This result includes the result \cite{cdr} of the 
existence of infinitesimal expansions for nonconvex 
polygons.

\begin{corollary}[Theorem 3 of {\cite[p. 215]{cdr}}]
If $\{\p_i\}$ is a nonconvex simple polygon with no 
straight verticies, then there exist $\{\vv_i\}$ satisfying:
\begin{align}
(\p_i-\p_{i+1})\cdot(\vv_i-\vv_{i+1})&=0\\
(\p_i-\p_j)\cdot(\vv_i-\vv_j)&>0\text{ for $i$ and $j$ not adjacent}
\end{align}
\end{corollary}

\begin{proof}
Apply Corollary \ref{cgen} to $\f=\text{the polygon}$ and:
\begin{equation}
\begin{split}
V=\{(x,y)\in(\R/2\pi)^2:\text{ }&\text{there are two full edges 
separating $x$ and $y$}\cr&\text{in both directions}\}
\end{split}
\end{equation}
Then we have a $\varphi$.  Set $\vv_i=\varphi(\f^{-1}(\p_i))$.  
Then $(\p_i-\p_j)\cdot(\vv_i-\vv_j)>0$ for $i$ and $j$ not 
adjacent is clear from (\ref{expan}).  Now:
\begin{equation}
(\vv_{i+1}-\vv_i)\cdot(\p_{i+1}-\p_i)=
\int_{\f^{-1}(\p_i)}^{\f^{-1}(\p_{i+1})}\!\varphi'(x)\cdot(\p_{i+1}-\p_i)\,dx=0
\end{equation}
since $\varphi\in Q_\f$.
\end{proof}

Theorem \ref{gen}, the main result of the generalization of 
the program of \cite{cdr} is essentially Corollary \ref{cgen} 
made uniform over some suitable set of curves.

\begin{theorem}[Analogue of Theorem 3 of {\cite[p. 215]{cdr}}]\label{gen}
Suppose $D_1\subset D$ is (weakly) closed and contains no 
convex curves, and that $V\subset D_1\cross (\R/2\pi)^2$ 
is closed.  Additionally, suppose that for every $(\f,x,y)\in V$, 
the line segment joining $\f(x)$ and $\f(y)$ is not completely 
contained in $\f(\R/2\pi)$.  Then there exists $\epsilon>0$ such 
that for each $\f\in D_1$, there exists $\varphi\in Q_\f$ with:
\begin{itemize}
\item[(1)] $\|\varphi\|=1$.
\item[(2)] $R_\f\varphi(x,y)\geq\epsilon$ whenever $(\f,x,y)\in V$.
\end{itemize}
\end{theorem}

We can see that Corollary \ref{cgen} is obtained by taking $D_1$ 
to consist of a single curve.  A corollary which does not lose 
the uniformity is the following:

\begin{corollary}
Suppose $D_1\subset D$ is weakly closed, contains no convex 
curves, and contains no curves with straight sections.  Then 
for every $\delta>0$, there exists $\epsilon>0$ such that for 
every $\f\in D_1$, there exists $\varphi\in Q_\f$ satisfying:
\begin{itemize}
\item[(1)] $\|\varphi\|=1$.
\item[(2)] $(\varphi(x)-\varphi(y))\cdot(\f(x)-\f(y))\geq\epsilon$ if $|x-y|\geq\delta$.
\end{itemize}
\end{corollary}

\begin{proof}
Choose $V=D_1\cross\{(x,y)\in(\R/2\pi)^2:|x-y|\geq\delta\}$ and apply 
Theorem \ref{gen}.
\end{proof}

The main difficulty in showing the existence of a $\varphi$ 
which is expansive for {\it all} pairs $x$ and $y$ 
is the fact $V$ being closed is critical to the proof.  Clearly 
$(\f,x,x)$ can never be in $V$ since then we would conclude 
that $(\varphi(x)-\varphi(x))\cdot(\f(x)-\f(x))>0$.  Hence, we 
must always exclude a neighborhood of the ``diagonal'' of 
$(\R/2\pi)^2$.  This means that we will not have shown that 
$(\varphi(x)-\varphi(y))\cdot(\f(x)-\f(y))>0$ for all 
pairs $x$ and $y$.

The following theorem is the essence of why expansive variations 
exist.  It relies on the generalization of the Maxwell-Cremona 
Theorem (Theorem \ref{maxwell}).

\begin{theorem}[Analogue of Theorem 4 of {\cite[p. 216]{cdr}}]\label{needmax}
If $\f\in D$ and $t\in T$ such that 
$\langle t,R_\f\alpha\rangle=0$ for all $\alpha\in Q_\f$, then 
either:
\begin{itemize}
\item[(1)] The curve $\f$ is convex.\newline OR
\item[(2)] For all $(x,y)\in\operatorname{supp}t$, the line 
segment connecting $\f(x)$ and $\f(y)$ 
is completely contained in $\f(\R/2\pi)$.
\end{itemize}
\end{theorem}

In the spirit of the generalization of the Farkas Lemma (Theorem 
\ref{farkas}), it is possible to prove that Theorem \ref{needmax} implies 
Theorem \ref{gen}.

\begin{proposition}[Analogue of Lemma 3 of {\cite[p. 216]{cdr}}]\label{impl}
Theorem \ref{needmax} implies Theorem \ref{gen}.
\end{proposition}

\subsection{Proof of Theorem \ref{needmax}}

\begin{proof}
Suppose that we have some $\f\in D$ and $t\in T$ 
with $\langle t,R_\f\alpha\rangle=0$ for all $\alpha\in Q_\f$.  
Let $\mathbf{\hat f}'$ denote $\f'/|\f'|$.

First, let us show that there exists 
$\beta\in\Ltwo(\R/2\pi)$ such that:
\begin{equation}
\langle t,R\alpha\rangle=\int_{\R/2\pi}\beta(x)\mathbf{\hat f}'(x)\cdot\alpha'(x)\,dx
\end{equation}
Clearly there exists $\mu\in\Ltwo(\R/2\pi,\R^2)$ 
such that 
$\langle t,R\alpha\rangle=\int_{\R/2\pi}\mu(x)\cdot\alpha'(x)\,dx$.  
Now we can constrain $\mu$ as follows.  For any 
$\lambda\in\Ltwo(\R/2\pi)$ satisfying 
$\int\lambda\mathbf{\hat f}'=0$, we know that:
\begin{equation}
\int_{\R/2\pi}\left(\mu(x)\cdot i\mathbf{\hat f}'(x)\right)\lambda(x)\,dx=0
\end{equation}
The set $H$ of such $\lambda$ is of codimension $2$ in 
$\Ltwo(\R/2\pi)$.  Now 
$\mu(x)\cdot i\mathbf{\hat f}'(x)\in H^\perp$, 
which is of dimension $2$.  But we can exercise 
two dimensions of freedom by 
adding constants to $\mu(x)$.  Thus we can assume 
$\mu(x)\cdot i\mathbf{\hat f}'(x)\equiv 0$, 
in other words $\mu\parallel\f'$, and hence is 
of the form $\beta(x)\mathbf{\hat f}'(x)$.

We will consider the operators $A_1,A_2\in\Lin(C_0(\R^2,\R^2),\R^2)$ 
defined by:
\begin{align}
A_1\U&:=\iint_{(\R/2\pi)^2}t(x,y)(\f(x)-\f(y))
\int_{\f(y)}^{\f(x)}\U\cdot\dd s\\
A_2\U&:=\int_{\R/2\pi}\beta(x)\mathbf{\hat f}'(x)[\U(\f(x))\cdot\f'(x)]\,dx
\end{align}
Since $A_1$ and $A_2$ are linear combinations of projections, they 
are symmetric, that is there exist $a_j,b_j,e_j\in\Lin(C_0(\R^2,\R),\R)=C_0(\R^2)^*$ 
such that $A_j=\left(\smallmatrix a_j&b_j\cr b_j&e_j\endsmallmatrix\right)$.  
Then $A:=A_1-A_2=\left(\smallmatrix a&b\cr b&e\endsmallmatrix\right)$, 
where $a,b,e\in\Lin(C_0(\R^2,\R),\R)=C_0(\R^2)^*$.  We have:
\begin{equation}
\begin{split}
A_1\grad g&=\iint_{(\R/2\pi)^2}t(x,y)(\f(x)-\f(y))(g(\f(x))-g(\f(y)))\cr
&=\Bigl(\langle t,R(\e_1g(\f(\cdot)))\rangle,\langle t,R(\e_2g(\f(\cdot)))\rangle\Bigr)\cr
&=\int_{\R/2\pi}\beta(x)\mathbf{\hat f}'(x)[\grad g(\f(x))\cdot\f'(x)]\,dx=A_2\grad g
\end{split}
\end{equation}
Hence $A\grad g=0$ for all $g\in C_0^\infty(\R^2)$.

By the generalization of the Maxwell-Cremona Theorem, 
Theorem \ref{maxwell}, there exists a $c\in C_c(\R^2)$ 
such that we have (in the distributional sense):
\begin{equation}
A\U=\iint_{\R^2}\left(\begin{matrix}\hfill c_{yy}&-c_{xy}\cr-c_{xy}&\hfill c_{xx}\end{matrix}\right)\U\,dx\,dy
\end{equation}
Now the matrices $\left(\begin{smallmatrix}\hfill c_{yy}&-c_{xy}\cr
-c_{xy}&\hfill c_{xx}\end{smallmatrix}\right)$ and 
$\left(\begin{smallmatrix}c_{xx}&c_{xy}\cr
c_{xy}&c_{yy}\end{smallmatrix}\right)$ are 
related by a similarity transform.  The former is 
a positive linear combination of projections at 
every point in $\R^2-\f(\R/2\pi)$, hence the latter 
is positive at every point not on the curve as well.  Hence 
$c$ is locally convex on the interior of the curve 
and on the exterior of the curve.

Now let $M=\sup_{\p\in\R^2}c(\p)$ and define the nonempty 
closed set $S=\{\p\in\R^2:c(\p)=M\}$.  
Suppose $\p\in\partial S$ and $\p\notin\f(\R/2\pi)$.  
Then there is a neighborhood of $\p$ which is disjoint 
from $\f(\R/2\pi)$.  In this neighborhood, $c$ will 
be convex.  Hence the whole neighborhood will belong 
to $S$, a contradiction.  Thus 
$\partial S\subseteq\f(\R/2\pi)$.  We thus have four 
cases:
\begin{itemize}
\item[(1)] $S$ is the closure of the exterior of the curve.
\item[(2)] $S$ is the closure of the interior of the curve.
\item[(3)] $S$ is a closed subset of the curve.
\item[(4)] $S$ is the whole plane.
\end{itemize}
If (1) is true, then $c$ is zero on the curve.  This 
implies that $\f$ is a level curve of a function with positive 
hessian and as such must be convex.  If (4) is true, 
then $c\equiv 0$.  Then for every $(x,y)\in\operatorname{supp}t$, 
we will necessarily have the line segment joining 
$\f(x)$ and $\f(y)$ completely contained in $\f(\R/2\pi)$.  
This is because if not, then there would be a point 
in $\R^2-\f(\R/2\pi)$ where the matrix
$\left(\begin{smallmatrix}\hfill c_{yy}&-c_{xy}\cr
-c_{xy}&\hfill c_{xx}\end{smallmatrix}\right)$ would be 
positive, giving $c$ upward convexity.  The case (2) is 
easily disposed of since $c=0$ outside the convex hull 
of the curve and hence will be zero on 
at least one point of the curve.  Hence the maximum 
value $c$ attains is zero, a contradiction.  Thus it 
suffices to show that case (3) cannot happen.

Assume (3) is true.  We have two cases:
\begin{itemize}
\item[(1$'$)] There exists $x\in\R/2\pi$ such that for every 
$\delta>0$, $\f([x,x+\delta])\nsubseteq S$ and $\f([x-\delta,x])\nsubseteq S$.
\item[(2$'$)] There does not exist such an $x\in\R/2\pi$.
\end{itemize}
I will deal with the easier case (1$'$) first.  WLOG $x=0$.  
Also, WLOG, $\f(x)\cdot\e_1$ is one to one for $|x|<\epsilon$.  
Choose $\delta_1,\delta_2>0$ such that the curve in the square 
$[-\delta_1,\delta_1]\cross[-\delta_2,\delta_2]\subset\R^2$ looks like the graph of 
a function, that is, $\f^{-1}([-\delta_1,\delta_1]\cross[-\delta_2,\delta_2])\subseteq[-\epsilon,\epsilon]$.  
Let $-\delta_1<x_-<0<x_+<\delta_1$ have $c(x_-,0)\ne M$ and 
$c(x_+,0)\ne M$.
Now let:
\begin{equation}
M'=\frac 12\left(M+\max_{\p\in\partial[x_-,x_+]\cross[-\delta_2,\delta_2]}c(\p)\right)<M
\end{equation}
Let $y_+$ be the least $y>0$ such that $c(0,y)=M'$ and let 
$y_-$ be the highest $y<0$ such that $c(0,y)=M'$.  Consider 
the level curves passing through $y_+$ and $y_-$.  By the 
convexity of $c$ they must curve away from $(0,0)$ where 
the maximum occurs, but they must meet the curve on both 
sides of $(0,0)$ at some $x_-'$ and $x_+'$.  This is a 
contradiction.

Now suppose (2$'$) is true.  Let $[x,y]\subset\R/2\pi$ 
satisfy $c(\f([x,y]))=M$ and for every $\delta>0$, 
$\f([x-\delta,x])\nsubseteq S$ and $\f([y,y+\delta])\nsubseteq S$.  
Then $\f([x,y])$ is a level curve of $c$ restricted to 
the interior of the curve.  As the level curve of a 
convex function it must be curved towards the interior 
of the curve.  But by the same reasoning, $\f([x,y])$ 
is a level curve of $c$ restricted to the outside of 
the curve, and hence must be curved towards the outside 
of the curve.  Hence $\f([x,y])$ is a line segment.  
As above, we can rotate $\f$ so it looks like the graph 
of a function $\R\map\R$ near $\f(x)$ and near $\f(y)$.  Using 
the same procedure as above, we get a contradiction 
by considering level curves of $M-\eta$ for a suitably small $\eta>0$.
\end{proof}

We have now justified every step in the proof of Theorem 
\ref{gen} except for Proposition \ref{impl} and the generalized 
Maxwell-Cremona Theorem.  We will prove these next.

\section{A Generalization of the Farkas Lemma}\label{farkass}

The Farkas Lemma from linear programming is as follows:

\begin{lemma}[Farkas Lemma]\label{farkasl}
Let $A:\R^n\to\R^m$ be a linear transformation.  Then exactly 
one of the following two statements holds:
\begin{itemize}
\item[(1)] There exists a nonzero $y\in\R^m$ whose components 
are all nonnegative and which satisfies $A^{\operatorname{T}}y=0$.
\item[(2)] There exists an $x\in\R^n$ such that every component 
of $Ax$ is positive.
\end{itemize}
\end{lemma}

The generalization of the Farkas Lemma that we 
will need will have the basic form:

\begin{theorem}\label{farkas}
Let $X$ be a compact Hausdorff space and $Y$ a (real) Hilbert space.  
Let $A:Y\map C(X)$ be linear and bounded.  Also let $A':C(X)^*\map Y$ 
denote its adjoint, that is 
$\langle\lambda,Ay\rangle=\langle A'\lambda,y\rangle$.  Then 
exactly one of the following two statements holds:
\begin{itemize}
\item[(1)] There exists a nonzero positive $t\in C(X)^*$ such that $A't=0$.
\item[(2)] There exists a $y\in Y$ such that $Ay>0$.
\end{itemize}
\end{theorem}

We remark that if we take $Y$ to be finite dimensional and $X$ to consist 
of a finite number of points, then we recover Lemma \ref{farkasl}.

\begin{proof}
It is trivial that (1) and (2) cannot simultaneously hold, 
for if so, $0=\langle A't,y\rangle=\langle t,Ay\rangle>0$.

It remains to show that $\sim$(1)$\implies$(2).  Let 
$T:=\{t\in C(X)^*:t\geq 0\}$.

I claim that there exists $\epsilon>0$ such that 
$\|A't\|\geq\epsilon\|t\|$ for all $t\in T$.  If we 
suppose the contrary, then there exists a sequence $t_n\in T$ with 
$\|t_n\|=1$ such that $A't_n\to 0$.  By the Banach-Alaoglu 
Theorem, there exists a subnet $t_\alpha$ which converges to 
$t\in T$ (in the weak-$*$ topology on $T$).  We know that we will 
have $t\in T$ and $\|t\|=1$.  Also, for all $y\in Y$, we have:
\begin{equation}
0=\lim_\alpha\langle A't_\alpha,y\rangle=
\lim_\alpha\langle t_\alpha,Ay\rangle=\langle t,Ay\rangle=\langle A't,y\rangle
\end{equation}
Thus $A't=0$, contradicting $\sim$(1).  Thus the claim is 
true.  I now can show (2).

Let $t_n\in T$ be a sequence such that $\|t_n\|=1$ and:
\begin{equation}
\|A't_n\|\to\inf_{\begin{smallmatrix}t\in T\cr\|t\|=1\end{smallmatrix}}
\|A't\|=:w\geq\epsilon
\end{equation}
Then a subnet $t_\alpha$ will converge in the weak-$*$ 
topology to a limit $t_{\infty}$.  Now:
\begin{equation}
w\leq\|A't_{\infty}\|\leq\liminf_\alpha\|A't_\alpha\|=w
\end{equation}
Hence $\|A't_{\infty}\|=w$.

Let $y:=A't_{\infty}/\|A't_{\infty}\|$.  I claim that 
$(Ay)(x)\geq\epsilon$ for all $x\in X$.  It suffices to 
show that $\langle t,Ay\rangle\geq w$ for all $t\in T$ 
with $\|t\|=1$.  But if $\langle t,Ay\rangle<w$ for some 
$t\in T$ with $\|t\|=1$, then 
$\langle A't,y\rangle<w$.  Consider then:
\begin{equation}
\begin{split}
\left.\frac d{d\eta}\right|_{\eta=0}
&\|A'((1-\eta)t_{\infty}+\eta t)\|^2\cr
&=\left.\frac d{d\eta}
\left[(1-\eta)^2\|A't_{\infty}\|^2+2\eta(1-\eta)\langle A't,A't_{\infty}\rangle
+\eta^2 \|A't\|^2\right]\right|_{\eta=0}\cr
&=-2w^2+2\langle A't,wy\rangle<0
\end{split}
\end{equation}
This is a contradiction since $\|(1-\eta)t_{\infty}+\eta t\|=1$.  
Hence the proof is complete.
\end{proof}

We can prove Proposition \ref{impl} using the same proof outline 
from Theorem \ref{farkas}.  We will, however, need the following 
approximation lemma.

\begin{lemma}\label{approx}
Suppose $\f_n\to\f$ in $D$ and that $q\in Q_\f$ is of the form 
$q'=\lambda i\f'$ where $\lambda$ is smooth.  Then there exist 
$q_n\in Q_{\f_n}$ such that $q_n\to q$ (weakly).
\end{lemma}

\begin{proof}
We will search for $q_n$ of the form $q_n'=(\lambda+\nu_n)i\f_n'$.  
We will have $\|q_n\|$ bounded if $\|\nu_n\|_\infty$ is bounded.  
Hence we will have $q_n\to q$ weakly if $\|\nu_n\|_\infty$ is 
bounded and $\langle\ell,q-q_n\rangle\to 0$ for all smooth 
$\ell\in H$.  Now $|\langle\ell,q-q_n\rangle|$ is equal to:
\begin{equation}
\begin{split}
\left|\int_{\R/2\pi}\ell'\cdot(q'-q_n')\right|&=\left|\int_{\R/2\pi}\ell'\cdot
(\lambda i\f'-\lambda i\f_n'-\nu_n i\f_n')\right|\cr
&\leq\left|\int_{\R/2\pi}\ell'\lambda\cdot
i(\f'-\f_n')\right|+\left|\int_{\R/2\pi}\nu_n\ell'\cdot i\f_n'\right|\cr
&=\left|\int_{\R/2\pi}[\ell''\lambda+\ell'\lambda']\cdot
i[\f-\f_n]\right|+\left|\int_{\R/2\pi}\nu_n\ell'\cdot i\f_n'\right|\cr
&\leq 2\pi\|\ell''\lambda+\ell'\lambda'\|_\infty\|\f-\f_n\|_\infty
+\|\nu_n\|_\infty\|\ell'\|_\infty\sqrt{2\pi}\|\f_n\|
\end{split}
\end{equation}
By Lemma \ref{unif2}, $\|\f-\f_n\|_\infty\to 0$.  Thus in order 
for $q_n\to q$ weakly, all we need is $\|\nu_n\|_\infty\to 0$ and 
$\int_{\R/2\pi}(\lambda+\nu_n)\f_n'=0$ (because clearly we 
must have $\int_{\R/2\pi}q_n'=0$).  Using integration 
by parts, this last equality can be written:
\begin{equation}\label{needforc}
\int_{\R/2\pi}\f_n\nu_n'=\int_{\R/2\pi}[\f-\f_n]\lambda'
\end{equation}
We can pick $a_1$, $a_2$, and $a_3$ in $\R/2\pi$ such that:
\begin{equation}
\left|\begin{matrix}1&1&1\cr\f(a_1)\cdot\e_1&\f(a_2)\cdot\e_1&\f(a_3)\cdot\e_1
\cr\f(a_1)\cdot\e_2&\f(a_2)\cdot\e_2&\f(a_3)\cdot\e_2\end{matrix}\right|\geq 2\epsilon>0
\end{equation}
There exists an $N$ such that for every $n\geq N$, the determinant 
with $\f$ replaced with $\f_n$ is greater than $\epsilon$.  It 
suffices to choose $\nu_n$ for $n\geq N$.  Set 
$C_n=\int_{\R/2\pi}[\f-\f_n]\lambda'$.  We know that 
$|C_n|\leq 2\pi\|\lambda'\|_\infty\|\f-\f_n\|_\infty$.  We solve 
the following system of equations for $b_{n,i}\in\R$:
\begin{align}
\hphantom{\f_n(a_1)}b_{n,1}+\hphantom{\f_n(a_2)}b_{n,2}+\hphantom{\f_n(a_3)}b_{n,3}&=0\\
\f_n(a_1)b_{n,1}+\f_n(a_2)b_{n,2}+\f_n(a_3)b_{n,3}&=C_n
\end{align}
For $n\geq N$, we can use Cramer's Rule to give the follwing 
bound on the solution:
\begin{equation}
|b_{n,i}|\leq\epsilon^{-1}2[2\pi\|\lambda'\|_\infty\|\f-\f_n\|_\infty]2[\sqrt{2\pi}\|\f_n\|]
\end{equation}
Set $\nu_n(0)=0$ and:
\begin{equation}
\nu_n'(x)=b_{n,1}\delta(x-a_1)+b_{n,2}\delta(x-a_2)+b_{n,3}\delta(x-a_3)
\end{equation}
Then we will guarantee $\int_{\R/2\pi}\nu_n'=0$, equation 
(\ref{needforc}), and $\|\nu_n\|_\infty\to 0$.  Thus we 
will have $q_n\to q$ (weakly).
\end{proof}

\begin{proof}[Proof of Proposition \ref{impl}]
We will write $V(\f)$ for 
$\{(x,y)\in(\R/2\pi)^2:(\f,x,y)\in V\}$.  Also, if 
$Z\subset(\R/2\pi)^2$, we will write $T_Z$ for 
$\{t\in C((\R/2\pi)^2):t\geq 0\text{ and }\supp t\subseteq Z\}$.  
We assume Theorem \ref{needmax}.  Let 
$\pi_\f:H\map Q_\f$ be the orthogonal projection and let 
$J_\f=\pi_\f\circ R_\f'$.  Then Theorem \ref{needmax} implies 
``If $\f\in D_1$, $t\in T_{V(\f)}$, and $J_\f t=0$, then 
$t=0$''.

I claim that there exists $\epsilon>0$ such that 
$\|J_\f t\|\geq\epsilon\|t\|$ for 
all $\f\in D_1$ and $t\in T_{V(\f)}$.  If we suppose 
the contrary, then there exist two sequences, $\f_n\in D_1$ 
and $t_n\in T_{V(\f_n)}$ with $\|t_n\|=1$ 
such that $\|J_{\f_n}t_n\|\to 0$.  Since 
$D_1$ is weakly closed, it is compact by the 
Banach-Alaoglu Theorem, hence there exists a convergent 
subsequence of $\f_n$ which we assume WLOG is the 
whole sequence, so that $\f_n\to\f$.  Since 
this means that $\f_n\to\f$ uniformly, we will have 
$\|R'_{\f_n}-R'_\f\|\to 0$.  Thus:
\begin{equation}
\|\pi_{\f_n}R'_{\f_n}t_n\|\to 0\implies\|\pi_{\f_n}R'_\f t_n\|\to 0
\end{equation}
Now there is also a weak-$*$ convergent subsequence of the 
$t_n$ by the Banach-Alaoglu Theorem, which again WLOG 
is the whole sequence.  Thus $t_n\to t\in T_{V(\f)}$ 
since $V$ is closed; also $\|t\|=1$.  Pick some $q\in Q_\f$ 
which can be written as $q'=\lambda i\f'$ where $\lambda$ 
is smooth (such $q$ are dense in $Q_\f$).  Let $q_n\in Q_{\f_n}$ 
be the sequence guaranteed to exist by Lemma \ref{approx}.  We note 
that since $q_n$ is weakly convergent, it is bounded.  Now:
\begin{equation}\label{feq1}
0=\lim_{n\to\infty}\langle\pi_{\f_n}R'_\f t_n,q_n\rangle
=\lim_{n\to\infty}\langle R'_\f t_n,q_n\rangle
=\lim_{n\to\infty}\langle t_n,R_\f q_n\rangle
\end{equation}
Now by Lemma \ref{unif2}, $R_\f q_n\to R_\f q$ strongly.  
Thus the final limit in equation (\ref{feq1}) is equal to 
$\langle t,R_\f q\rangle$.  This means that $\langle R_\f't,q\rangle=0$ 
for a dense subset of $q\in Q_\f$.  Thus $J_\f t=0$ where 
$\f\in D_1$ and $t\in T_{V(\f)}-\{0\}$, contradicting 
Theorem \ref{needmax}.  Thus the claim is proved.

We can now show the existence of an appropriate $\varphi$ 
for every $\f\in D_1$ exactly as in the proof of Theorem 
\ref{farkas}.

Fix some $\f\in D_1$.  Let $t_n\in T_{V(\f)}$ be a 
sequence such that $\|t_n\|=1$ and:
\begin{equation}
\|J_\f t_n\|\to\inf_{\begin{smallmatrix}t\in T_{V(\f)}\cr\|t\|=1\end{smallmatrix}}\|J_\f t\|
=:w\geq\epsilon
\end{equation}
A subsequence is weak-$*$ convergent (WLOG the whole 
sequence) to a limit $t_{\infty}$.  Using the same reasoning 
as above, we conclude that $J_\f t_n\to J_\f t_{\infty}$ in 
the weak topology, so:
\begin{equation}
w\leq\|J_\f t_{\infty}\|\leq\liminf\|J_\f t_n\|=w
\end{equation}
Thus $\|J_\f t_{\infty}\|=w$.  Let $q:=J_\f t_{\infty}/\|J_\f t_{\infty}\|$.

Now I claim that $\langle J_\f t,q\rangle\geq w\|t\|$ 
for all $t\in T_{V(\f)}$.  Suppose not, that we have $t\in T_{V(\f)}$ 
with $\|t\|=1$ and $\langle J_\f t,q\rangle<w$.  
Then $\langle J_\f t,J_\f t_{\infty}\rangle<w^2$.  But 
consider then:
\begin{equation}
\begin{split}
\left.\frac d{d\eta}\right|_{\eta=0}
&\|J_\f((1-\eta)t_{\infty}+\eta t)\|^2\cr
&=\left.\frac d{d\eta}
\left[(1-\eta)^2\|J_\f t_{\infty}\|^2+2\eta(1-\eta)\langle J_\f t,J_\f t_{\infty}\rangle
+\eta^2 \|J_\f t\|^2\right]\right|_{\eta=0}\cr
&=-2w^2+2\langle J_\f t_{\infty}, J_\f t\rangle<0
\end{split}
\end{equation}
This is a contradiction since $\|(1-\eta)t_{\infty}+\eta t\|=1$.  
Hence the claim is proved.

Let $\varphi=q$.  Then:
\begin{equation}
\langle t,R_\f\varphi\rangle=\langle J_\f t,q\rangle\geq w\|t\|\geq\epsilon\|t\|
\text{ for all $t\in T_{V(\f)}$}
\end{equation}
This means that $R_\f\varphi(x,y)\geq\epsilon$ for all $(x,y)\in V(\f)$.
\end{proof}

\section{A Generalization of the Maxwell-Cremona Theorem}\label{maxwells}

Let $A\in\Lin(C_0(\R^2,\R^2),\R^2)$ 
have compact support.  Then by the Riesz Representation Theorem, 
$A$ can be thought of as a matrix of measures on $\R^2$:
\begin{equation}
A=\left(\begin{matrix}a&b\cr d&e\end{matrix}\right)
\end{equation}
We are concerned with the case when $A$ is symmetric, that 
is $b=d$.  For the moment, suppose 
$a$, $b$, and $e$ are continuous functions.  In this case, 
at each point $A$ has orthogonal eigenvectors $\vv_1$ and $\vv_2$ 
with eigenvalues $\lambda_1$ and $\lambda_2$.  We think of 
$A$ as representing a ``stress'' on the plane, where at each 
point, there is tension in the $\vv_i$ direction of magnitude 
$\lambda_i$.  It turns out that it is right to call such a 
stress is an ``equilibrium stress'' if:
\begin{equation}\label{eqstress}
A\grad g=0\text{ for all $g\in C_0^\infty(\R^2)$}
\end{equation}
In the case that $a$, $b$, and $e$ are continuous, it is 
straightforward to show that in fact:
\begin{equation}\label{maxc}
A=\left(\begin{matrix}a&b\cr b&e\end{matrix}\right)=
\left(\begin{matrix}\hfill c_{yy}&-c_{xy}\cr-c_{xy}&\hfill c_{xx}\end{matrix}\right)
\end{equation}
The function $c$ will be in $C_c(\R^2)$.  This is the 
Maxwell-Cremona ``lifting'' of the stress represented by 
$A$.

However, the notion of being an equilibrium stress (\ref{eqstress}) 
makes sense for any compactly supported $A$, so one would 
expect that (\ref{maxc}) should hold in some sense for all 
equilibrium stresses $A$.  If $\U$ is a smooth vector field 
and we integrate 
$\iint_{\R^2}\left(\begin{smallmatrix}\hfill c_{yy}&-c_{xy}\cr
-c_{xy}&\hfill c_{xx}\end{smallmatrix}\right)\U\,dx\,dy$ by parts, we get 
$\iint_{\R^2}c[i\grad\curl\U]\,dx\,dy$, so if (\ref{maxc}) holds 
in the distributional sense, we would like this last integral 
to give $A\U$ for smooth $\U$.  This is the intuition for 
the following theorem.

\begin{theorem}\label{maxwell}
Let $A\in\Lin(C_0(\R^2,\R^2),\R^2)$ have compact support.  
Suppose $A$ is symmetric, that is there exist 
$a,b,c\in C_0(\R^2)^*$ such that:
\begin{equation}
A=\left(\begin{matrix}a&b\cr b&e\end{matrix}\right)
\end{equation}
Additionally, suppose that for every $g\in C_0^\infty(\R^2)$, 
$A\grad g=0$.  Then there exists $c\in C_c(\R^2)$ such that 
for all $\U\in C_0^\infty(\R^2,\R^2)$:
\begin{equation}\label{mconc}
A\U=\iint_{\R^2}c[i\grad\curl\U]\,dx\,dy
\end{equation}
\end{theorem}

\begin{proof}
First, let us show that (the matrix of measures associated with) 
$A$ has no pure point part.  Let $\p$ 
and $\vv$ be arbitrary.  Choose $g\in C_0^\infty(\R^2)$ so 
that $\grad g(\p)=\vv$.  Then $0=A(\grad g)(\epsilon(\cdot)+\p)$, 
but as $\epsilon\to 0$, right hand side approaches the pure 
point part of $A$ at $\p$ applied to $\vv$.  Hence $A$ has 
no pure point part.

Consider the measure $|A|\in C_0(\R^2)^*$, where the 
$|\cdot|$ of a matrix is its operator norm.  In other words, 
for $f\geq 0$, we define:
\begin{equation}
|A|f:=\sup_{\begin{smallmatrix}\theta:\R^2\map\R\cr\psi:\R^2\map\R\end{smallmatrix}}\iint_{\R^2}
\left(\begin{matrix}\cos\theta&\sin\theta\end{matrix}\right)
\left(\begin{matrix}a&b\cr b&e\end{matrix}\right)
\left(\begin{matrix}\cos\psi\cr \sin\psi\end{matrix}\right)f
\end{equation}
We know $|A|$ comes from a measure, which we will also denote 
$|A|$.  Let $\mu(\theta)$ 
be the measure on the real line $\R$ at angle $\theta$ passing 
through the origin, obtained by projecting the 
measure $|A|$ orthogonally onto the line.  In other words:
\begin{equation}
\int_\R f(x)\,d\mu(\theta)=\iint_{\R^2}f((x,y)\cdot(\cos\theta,\sin\theta))|A|
\end{equation}
Now let $\mu_{\text{pp}}(\theta)$ be the pure point part of 
$\mu(\theta)$.  I claim that $\mu_{\text{pp}}(\theta)\ne 0$ for 
at most countably many $\theta$.  We note that this is implied 
by the following:
\begin{equation}\label{crit}
\sum_{i=1}^N\|\mu_{\text{pp}}(\theta_i)\|\leq\||A|\|\text{ whenever $\theta_i$ are distinct}
\end{equation}
But (\ref{crit}) is true because any part of $|A|$ which 
contributes to both $\|\mu_{\text{pp}}(\theta_i)\|$ and 
$\|\mu_{\text{pp}}(\theta_j)\|$ would have to be supported on 
a countable set of points, and hence would have to be pure point, 
which we know $A$, and hence $|A|$ does not have.  Now let 
$m(\theta,h)=\sup_{x\in\R}\int_x^{x+h}\mu(\theta)(y)\,dy$.  
Now $m(\theta,h)\to 0$ as $h\to 0$ if $\mu(\theta)$ has no 
pure point part, thus $m(\theta,h)\to 0$ for almost all $\theta$.  
This fact being proved, we can proceed to the construction of $c$.

Let $\phi$ be a smooth real valued even function on 
$\R^2$ with support contained in the unit disc which satisfies 
$\phi\geq 0$ and $\iint_{\R^2}\phi=1$.  Let 
$\phi_\eta(\p)=\eta^{-2}\phi(\eta^{-1}\p)$.  We can then define 
the operator:
\begin{equation}
A_\eta=A*\phi_\eta=\left(\begin{matrix}
a^{(\eta)}&b^{(\eta)}\cr b^{(\eta)}&e^{(\eta)}\end{matrix}\right)
\end{equation}

Now we know that:
\begin{equation}
a^{(\eta)},b^{(\eta)},e^{(\eta)}\in C_0^\infty(\R^2)\text{ and that 
$A_\eta\grad g=0$ for all $g\in C_0^\infty(\R^2)$}
\end{equation}
Thus the vector fields $(a^{(\eta)},b^{(\eta)})$ and 
$(b^{(\eta)},e^{(\eta)})$ have zero divergence.  That 
means there exist $f^{(\eta)},g^{(\eta)}\in C_0^\infty(\R^2)$ such 
that $a^{(\eta)}=f^{(\eta)}_y$, 
$b^{(\eta)}=-f^{(\eta)}_x=-g^{(\eta)}_y$, and 
$e^{(\eta)}=g^{(\eta)}_x$.  The equality $f^{(\eta)}_x=g^{(\eta)}_y$ implies 
that there exists $c^{(\eta)}\in C_0^\infty(\R^2)$ such 
that $f^{(\eta)}=c^{(\eta)}_y$ and $g^{(\eta)}=c^{(\eta)}_x$.  
In other words:
\begin{equation}
A_\eta=\left(\begin{matrix}\hfill c^{(\eta)}_{yy}&-c^{(\eta)}_{xy}
\cr-c^{(\eta)}_{xy}&\hfill c^{(\eta)}_{xx}\end{matrix}\right)
\end{equation}

{\bf Claim: For every $\epsilon>0$, there exist $\delta>0$ 
and $\eta_0>0$ such that:
\begin{equation}
\eta_0>\eta>0\text{ and }|\q-\p|<\delta\implies|c^{(\eta)}(\p)-c^{(\eta)}(\q)|<\epsilon
\end{equation}}

Let $\epsilon>0$ be given.  Suppose $\p=(x_0,y)\in\R^2$ and 
$\q\in\R^2$ and we wish to bound 
$|c^{(\eta)}(\p)-c^{(\eta)}(\q)|$ given $|\q-\p|<\delta$.  
To simplify notation, we will for the moment assume 
that $\q=(x,y)$.  Then:
\begin{equation}
\begin{split}
\left|c^{(\eta)}(\p)-c^{(\eta)}(\q)\right|&=\left|\int_{x_0}^xc^{(\eta)}_x(t,y)\,dt\right|
=\left|\int_{x_0}^x\int_{-\infty}^yc^{(\eta)}_{xy}(t,z)\,dz\,dt\right|\cr
&\leq\int_{x_0}^x\int_{-\infty}^\infty\left|b^{(\eta)}(t,z)\right|\,dz\,dt
\leq\int_{x_0-\eta}^{x+\eta}\int_{-\infty}^\infty\left|A(t,z)\right|\,dz\,dt\cr
&\leq m(0,\delta+2\eta)
\end{split}
\end{equation}
Similary, if $\theta_\p^\q$ is the angle of the segment from 
$\p$ to $\q$, then we have:
\begin{equation}
\left|c^{(\eta)}(\p)-c^{(\eta)}(\q)\right|\leq m(\theta_\p^\q,2\eta+\delta)
\end{equation}
Now since $m(\theta,h)\to 0$ as $h\to 0$ for all but at most 
countably many $\theta$, there exists $h>0$ such that the 
measure of the set $\{\theta:m(\theta,h)<\epsilon/4\}$ 
is more than $\frac{5\pi}3$.  Then if $2\eta+\delta<
\min(\epsilon/(4\pi\|t\|),h)$ and the slope the segment from 
$\p$ to $\q$ is not in the exceptional set of $\theta$ 
(which has measure less than $\frac\pi 3$), 
then $|c^{(\eta)}(\p)-c^{(\eta)}(\q)|\leq\epsilon/2$.  But 
for any $\p$ and $\q$ within $\delta$ of each other, 
we can find a $\rr$ within $\delta$ of both $\p$ and 
$\q$ so that neither of the segments $\p$ to $\rr$ and 
$\rr$ to $\q$ are in the exceptional set of $\theta$.  
Hence by the triangle inequality, 
$|c^{(\eta)}(\p)-c^{(\eta)}(\q)|\leq\epsilon$ if we set 
$\eta_0=\delta=\frac 14\min(\epsilon/(4\pi\|t\|),h)$.  
Thus the claim is true.

Now by the Arzel\`a-Ascoli Theorem, there exists a 
subsequence of $c^{(1/n)}$ which converges uniformly 
to a continous function $c\in C_c(\R^2)$.  Thus let 
$\eta_i\to 0$ and satisfy $c^{(\eta_i)}\to c\in C_c(\R^2)$ 
uniformly as $i\to\infty$.  As remarked before, if $\U$ 
is smooth compactly supported vector field, then it is 
a straightforward integration by parts to show:
\begin{equation}
A(\U*\phi_{\eta_i})=A_{\eta_i}\U=\iint_{\R^2}c^{(\eta_i)}[i\grad\curl\U]\,dx\,dy
\end{equation}
Taking the limit as $i\to\infty$, we obtain (\ref{mconc}) 
as was to be shown.
\end{proof}

\section{Open Problems}

Now, I can state some conjectures on possible strengthening 
of Theorem \ref{main}.  For example, we can 
conjecture that there exists an $\h$ which is not only 
continuous, but in fact smooth.  Also, if the initial curve is 
smooth, we can require that the curve be smooth at every 
time during the deformation.

\begin{conjecture}
Given a unit speed simple closed curve $\f:\R/2\pi\map\C$, 
there exists a smooth function $\h:[0,1]\map\mathcal D$ satisfying 
(1)--(4).
\end{conjecture}

\begin{conjecture}
Given a smooth unit speed simple closed curve $\f:\R/2\pi\map\C$, 
there exists a continuous function $\h:[0,1]\map\mathcal D$ satisfying 
(1)--(4) as well as:
\begin{itemize}
\item[(5)] $\h(t)(x)$ is a smooth function of $x$ for all $t\in[0,1]$.
\end{itemize}
\end{conjecture}

I also conjecture that it is possible to extend Corollary \ref{cgen} 
to something resembling the following.

\begin{conjecture}
Suppose $\f:\R/2\pi\map\R^2$ is a rectifiable simple closed 
curve which is not convex.  Then there exists 
$\varphi:\R/2\pi\map\R^2$ which is absolutely continuous and 
satisfies $\f'\cdot\varphi'\equiv 0$, as well as 
$(\f(x)-\f(y))\cdot(\varphi(x)-\varphi(y))>0$ whenever the 
line segment connecting $\f(x)$ and $\f(y)$ is not 
completely contained in $\f(\R/2\pi)$.
\end{conjecture}

Of course, this would be in preparation to prove:

\begin{mconjecture}
There exists a proof of Theorem \ref{main} which does not rely 
on approximation by polygons.
\end{mconjecture}

\bibliographystyle{amsplain}
\bibliography{unfoldARXIV}

\end{document}